\documentclass{amsart} 
\usepackage{amsmath,amscd} 
\usepackage{latexsym,amsmath,amssymb,mathrsfs,setspace,enumerate}
\usepackage[T1]{fontenc}
\usepackage[utf8]{inputenc}
\usepackage{lmodern} 
\usepackage{amssymb}

\usepackage{setspace}
\usepackage{graphics}
\usepackage{graphicx}
\usepackage[utf8]{inputenc}
\usepackage[english]{babel}
\usepackage{fancyhdr}

\theoremstyle{plain}

\newtheorem{prop}{Proposition} 
\newtheorem{thm}{Theorem}
\newtheorem{exam}{Example}

\newtheorem{defn}{Definition}

\newtheorem{note}[thm]{Note}
\newtheorem{res}[thm]{Result}

\title{Generalized Groups and Module Groupoids}
\author {P. G. Romeo and Sneha K K}
\address{Dept. of Mathematics, Cochin University of Science and Technology, 
Kochi, Kerala, INDIA.}
\email{$romeo_-parackal@yahoo.com,\, snehamuraleedharan007@gmail.com $}
\subjclass{20M10}
\keywords { Category, Functors, Groupoid, Generalized group, Generalized group 
groupoid Generalized module, Generalized module groupoid.}
\thanks{} 
\date{}

\begin{document}

\begin{abstract}
In this paper we discuss generalized group, provides some interesting examples. Further we introduce a generalized module as a module like structure obtained from a generalized group and discuss some of its properties and we 
also describes generalized module groupoids.

\end{abstract}
\maketitle

\section{Introduction}
The generalized groups introduced by Molaei is an interesting generalization of groups(cf.\cite{molai}).
The identity element in a group is unique, but in a generalized group there exists an identity forall
each elements. Clearly every group is a generalized group. A groupid is another generalization of a
group is a small category in which every morphism is invertible and was first defined by Brandt in the year 1926.
Groupoids are studied by many Mathematicians with different objective. One of the different approach is the structured
groupoid which is obtained by adding another structure in such a way that the added structure is compatible with the groupoid operation.
\par In this paper we introduce the module action on a generalized group and we call the resulting structure as a generalized module. Then 
we discuss some interesting examples and properties of generalized modules. Further, analogoues to generalized group groupoid we describe 
the generalized module groupoid  over a ring and obtained a relation between category of generalized module and category of generalized module groupoids.

\section{Preliminaries}
In this section we briefly recall all basic definitions and 
the elementary concepts needed in the sequel. In particular we recall the 
definitions of  categories, groupoids, generalized groups  with 
examples and  discuss some intersetng properties of these structures. 

\begin{defn} (cf.\cite{kss}) A Category ${\mathcal C} $ consists of the following data$:$ 
\begin{enumerate}
\item A class called the class of vertices or objects $\nu {\mathcal C}. $ 
\item A class of disjoint sets ${\mathcal C}(a,b)$ one for each pair $ (a,b) 
\in \nu {\mathcal C}\times \nu {\mathcal C}.$  An element $f \in{\mathcal C}$ is 
called a morphism from $ a $ to $ b, $ written $ f : a \rightarrow b$  ; $ a = 
dom\; f$ the domain of $f$  and $ b = cod\; f $  called the codomain of 
$f . $  
\item For $ a, b, c, \in \nu{\mathcal C}, $ a map
$$ \circ: {\mathcal C} ( a,b) \times {\mathcal C}( b , c)  \rightarrow {\mathcal C}( a , c) 
\,\,\text{given by}\,\, ( f , g )  \mapsto f \circ g $$
 is the composition of morphisms in $ {\mathcal C} . $
\item  for each $ a \in \nu {\mathcal C} $, a  unique $  1_a \in {\mathcal C}( 
a,a ) $ is  the identity morphism on $a.$
\end{enumerate}
These must satisfy the following axioms:
\begin{itemize}
\item(cat1) \qquad for $ f \in  {\mathcal C} ( a,b) 
, g \in  {\mathcal C}( b , c)\; and\; h \in   {\mathcal C}( c , d ), $ then  
 $$ f \circ ( g \circ h) = (f \circ g) \circ h $$
\item (cat 2)\qquad for each $ a \in \nu {\mathcal C}, f \in {\mathcal C}\; (a,b) \; 
and\; g \in {\mathcal C} (c , a),$ then 
$$1_a \circ f = f \qquad   and \qquad g \circ 1_a = g $$.
\end{itemize}
\end{defn}

Clearly $\nu \mathcal C$ can be identify as a subclass of $\mathcal C,$ and with 
this identification
it is possible to regard categories in terms of morphisms alone. The category 
$\mathcal C$ is said to
be small if the class $\mathcal C$ is a set. A morphism $f\in \mathcal C(a, b)$ 
is said to be an
isomorphism if there exists $f^{-1} \in \mathcal C(b, a)$ such that $ff^{-1} = 
1_a = e_a,$ domain identity
and $f^{-1}f = 1_b= f_b,$ range identity.
 \begin{exam}
 A group $G$ can be regarded as category in the following way; define category $\mathcal C$ with just one object 
 say $\nu \mathcal{C}=G$ and morphisma $\mathcal C = \{g: g\in G\}$ with composition in $\mathcal C$ is the binary 
operation in $G.$ Identity element in the group will be the identity morphism on  the 
vertex $G.$
 \end{exam}
  
\begin{defn}
A groupoid $\mathcal G=(\nu \mathcal G, \mathcal G)$ is a small category such that for morphisms 
$f,g\in \mathcal G$ with $cod\,f=dom\,g$ then $fg\in \mathcal G$ and every morphism is an isomorphism.  A 
groupoid $\mathcal G$ is said to be connected if for all $a\in \nu \mathcal G,\;\; \mathcal G(a, a) \neq \phi$
\end{defn}

\begin{exam}
Every group can be regarded as a groupoid with only one object.
\end{exam}
\begin{exam}
For a set $X$ the cartesian product $X\times X$ is a groupoid over $X$ with 
morphisms are the elements in $X\times X$
with the composition $ (x,y)\cdot(u,v)$ exists only when $y=u$ and is given by 
$(x,y)(u,v)=(x,v).$ In particular $(x,x)$
is the unique left identity and $(y, y)$ is the unique right identity.
\end{exam}
\begin{defn}
 
 Given two categories  $\mathcal C$ and $\mathcal D,$ a 
functor $F: \mathcal C \rightarrow \mathcal D$
 consists of two functions: the object function denoted by $\nu F$ which 
assigns to each object $a$ of $\mathcal C,$ an object
 $\nu F(a)$ of the category $\mathcal D$ and a morphism function which assigns 
to each morphism $f: a\rightarrow b$ of $\mathcal C,$
a morphism $F(f): F(a)\rightarrow F(b) $ in $\mathcal D  $ satisfies the 
following 
$$F(1_c) = 1_{F(c)}\,\text{for every}\, c \in \nu \mathcal C,\,\,\text{and}\,\, 
F(fg)= F(f) F(g)$$ 
whenever $fg$ is defined in $\mathcal C.$
\end{defn}

We denote by $\bf{Gpd}$ the category of groupoids in which objects are the 
groupoids and morphisms are the functors.\par

 A semigroup is a pair $(S, \cdot)$ where $S$ is a nonempty set
 and $\cdot$ is an associative binary operation on $S.$ A semigroup
 $S$ is said to be regular if for each $x\in S$ there exists $x'\in S$
 such that $xx'x= x.$  An element $e$ of $S$ is said to be idempotent if
 $e^2 = e.$ A band is a semigroup in which all elements are idempotents. Inverse 
 semigroup is a regular semigroup in which for each $x\in S$ there exists an 
unique $y\in S$ such that $x= xyx$ and $y = yxy.$

\subsection{Generalized Groups and Generalized group groupoids.\\} 
In the following  we recall the definitions of generalized groups and describe 
some of its properties. 

\begin{defn}(cf.\cite{molai})
 A generalized group $G$ is a non-empty set together with a binary operation 
called multiplication subject to the set of rules given below:
 \begin{enumerate}
  \item $(ab)c = a(bc)$ for all $a,\; b,\;c \;\;\in G $
  \item for each $a\in G$ there exists unique $e(a)\in G$ with $ae(a) = e(a)a= 
a$
  \item for each $a \in G$ there exists $a^{-1} \in G$  with $aa^{-1} = a^{-1}a 
= e(a)$
 \end{enumerate}
\end{defn}

It is seen that for each element $a$ in a generalized group
the inverse is unique and both $a$ and $a^{-1}$ have the same identity. Every 
abelian generalized group is a group.

\begin{defn}(cf.\cite{molai}) A generalized group $G$ is said to be normal 
generalized group if $ e(ab) = e(a) e(b)$for all elements $a,b \in G$ 
\end{defn} 
\begin{defn}(cf.\cite{molai})
 A non-empty subset $H$ of a generalized group $G$ is a generalized
subgroup of $G$ if and only if for all $a, b \in H,\;\; ab^{-1} \in H.$
\end{defn}

\begin{thm}
Rectangular  band semigroup is a normal generalized group. 
\end{thm}
\begin{proof}
A band $ B$ is a regular semigroup in which every element is an 
idempotent. Let $I,$ $\Lambda$ be non empty sets and $B = I\times \Lambda$
with multipication defined by 
$$ (i,\lambda)(j, \mu) = (i, \mu) \;\; for\;all \;\; (i,\lambda),(j, \mu)\in I\times\Lambda $$
is a semigroup and in which all elements are idempotents. Then $B$ is a rectangular band.  

For any $ (i, \lambda), (j, \mu) \in B $
$$  (i, \lambda)(j, \mu) = (j, \mu)(i, \lambda) = (i, \lambda)\;\; gives$$ 

$$ (j, \mu)= (i,\lambda)$$
 Thus we have $e(i, \lambda) = (i, \lambda)\; for\;all \; (i, \lambda) \in B$ and
  $(i, \lambda)^{-1} = (i, \lambda).$\\
  Moreover
  $$e((i, \lambda)(j, \mu)) = e(i, \lambda)e(j,\lambda)  $$
    hence $B$ is a normal generalized group.
\end{proof}
\begin{exam}
Let $G = (V(G), E(G))$ be a complete digraph without multiple edges, then each 
edge can be uniquely identified
with the starting and ending vertices. Let $d(g)$ stands for the domain and 
$r(g)$ stands for the range of the edge $g.$ 
For any $f, g \in E(G)$ define the composition of $f\; and \;g$ as the unique 
edge starting from the domain of $f$
and ending at the range of $g.$
$$ie; \quad f \circ g = h^{r(g)}_{d(f)}\;\;\; with \;\;d(h^{r(g)}_{d(f)}) 
=d(f)\;\; and 
\;\;r(h^{r(g)}_{d(f)}) = r(g).$$
The composition $\circ$ is an associative binary operation. For, $ f, g, h 
\in E(G)$ with $d(f) = v_1 \; \; r(f) = v_2,\,\,d(g) = u_1, \; \; r(g) = u_2,\,\,d(h) = w_1$ and $ r(h) = w_2$
 \begin{equation*}
 \begin{split}
 (f\circ g)\circ h &  = (h^{u_2}_{v_1}) \circ h  = h^{w_2}_{v_1}  \\
                   f\circ ( g\circ h )&  = f \circ(h^{w_2}_{u_1}) = h^{w_2}_{v_1}  \\
              ie.,\,     (f\circ g)\circ h &  =  f\circ ( g\circ h )
 \end{split}
 \end{equation*}
 also for each edge $f,$
 $$ f\circ f = f$$ hence $e(f) = f^{-1} = f$. Thus the complete digraph is a 
generalized group.
 \end{exam}
 
 \begin{exam}
 $G = \Bigg\{ A = $\(\begin{pmatrix}
a&b&\\
0&0&\\
\end{pmatrix}\) $ : a \neq 0,\;\; a,\; b\; \in R\Bigg\} $ \\
Then, $G$ is a generalized group and for all $A \in G$,\\\\
$e(A) =$ \(\begin{pmatrix}
1&b/a&\\
0&0&\\
\end{pmatrix}\) and $A^{-1} = $ \(\begin{pmatrix}
1/a&b/a^2&\\
0&0&\\
\end{pmatrix}\)\\\\
where $e(A)$ and $A^{-1}$ are the identity and the inverse of matrix $A$ respectively.
Also $e(AB) = e(B)$ for all $A, B \in G $
\end{exam}
 
\begin{defn}(cf.\cite{molai})
 Let $G$ and $H$ be two generalized groups. A generalized group homomorphism 
from $G$ to $ H $  is a map 
 $f: G\rightarrow H$ such that $$ f(ab) = f(a)f(b)$$
\end{defn}
				
\begin{thm}(cf.\cite{molai})
 Let $f: G\rightarrow H$ be a homomorphism of generalized groups 
$G$ and $H.$ Then
 \begin{enumerate}
  \item $f(e(a))= e(f(a))$ is an identity element in $H$ for all $a\in G $  
  \item $ f(a^{-1}) = f(a)^{-1}$
  \item if $K$ is a generalized subgroup of $G$ then $f(K)$ is a generalized 
subgroup of $H.$
   \end{enumerate}
\end{thm}

Generalized groups and their homomorphisms form a category and is denoted by 
$GG$

 \begin{defn}(cf.\cite{Gursoy})
  A generalized group groupoid $\mathcal G $ is a groupoid $(\nu\mathcal{G}, \mathcal{G})$ endowed 
  with the structure of generalized group such that the following  maps 
  \begin{enumerate}
  \item $+ : G\times G \rightarrow G, \;\; (f, g) \rightarrow f + g,$
 
  \item $u : G\rightarrow G,\;\; f \rightarrow -f,$ 
  \item $e:G\rightarrow G,\;\; f \rightarrow e(f),$
  \end{enumerate}
  are functorial.
  \end{defn}
  
  Since $+$ is a functorial we have 
  \begin{equation*}
   \begin{split}
   ( f\circ g) + (h\circ k) & = +(f\circ g, h\circ k)\\
      & = +[(f, h)\circ(g, k)] \\
      & = +(f, h) \circ +(g, k)\\
      & = (f+h) \circ (g+k)\\
      ( f\circ g) + (h\circ k) & = (f+h) \circ (g+k)
   \end{split}
  \end{equation*}

  thus the interchange law 
  $$(f\circ g)+(h\circ k) = (f+h) \circ (g+k)$$
  exists between groupoid composition and generalized group operation.
  
   In other words, a generalized group groupoid is a groupoid endowed with a 
structure of generalized group such that the structure maps of groupoid are generalized group homomorphisms.
  
 \begin{exam}
 Let $G$ be a generalized group. Then $G\times G$ is a generalized group 
groupoid with object set $G$.
 For each morphism $(x, y)\in G\times G$ the identity arrow of $(x,y)$ is $(e(x),e(y))$ and 
the inverse is $(-x, -y)$ and the
 interchange law also holds. For,   
 \begin{equation*}
  \begin{split}
   [(f,g)\circ(g, h)]+[(f', g')\circ(g', h')]& = (f, h)+(f', h')\\
                                    & = (f+f',h+h')\\
  [(f,g)+ (f',g')] \circ [(g, h)+(g', h')] & = (f+f', g+g')\circ (g+g', h+h')\\
  & = (f+f', h+h')\\
  [(f,g)\circ(g, h)]+[(f', g')\circ(g', h')]& = [(f,g)+ (f',g')] \circ 
[(g,h)+(g', h')]
  \end{split}
 \end{equation*}
\end{exam}

\subsection{Generalized Groups from a connected groupoid.\\}
 
We proceed to describe the generalized group obtained from a  
connected groupoid. Let $\mathcal G$ be a connected groupoid with $\mathcal G(a, 
b)$ 
contains exactly one morphism for all $a,\; b\;\in \mathcal G .$ 
Define a binary operation which is deoted by $+$ on $\mathcal G.$ For, 
let $f \in \mathcal G(a, b)$ and $g \in \mathcal G(c,d)$ 
 $$f+g = f\circ k_{bc}\circ g $$ where $\circ$ is the composition in the 
groupoid 
and 
$k_{bc}$ is the unique 
 morphism in $\mathcal G(b,c)$. Then $\mathcal G$ is a generalized group, since

\begin{enumerate}
\item The operation $"+"$ is associative.\\ for $f\in \mathcal G(a,b)\; g\in 
\mathcal G(c,d)\; h\in \mathcal G(i,j)$
\begin{equation*}
 \begin{split}
  (f+g)+h & = (f\circ k_{bc}\circ g)+h \\ 
  & = (f\circ k_{bc}\circ g)\circ( k_{di}\circ h ) \\ 
   & = f k_{bc}gk_{di} h\\
    f+(g+h) & = f+( g\circ k_{di}\circ h)\\
    & = (f\circ (k_{bc})\circ (g\circ k_{di}\circ h)\\
    & = f k_{bc}gk_{di} h
 \end{split}
\end{equation*}

$ie;$ $(f+g)+h = f+(g+h)$ and $+$ is an associative binary opration.
\item for each $f \in \mathcal G$
$$ f+f = f\circ f^{-1}\circ f = f$$
hence $e(f) =  -f = f$\\
Thus the groupoid $\mathcal G$ together with the operation $+,$ forms a 
generalized group and is denoted by $\mathcal G^*.$ More over this generalized group $\mathcal G^*$ is normal, since $ e(f+g) = f+g 
= e(f)+e(g)$ 
\end{enumerate}

Now we extend this construction of a generalized group in to an arbitrary 
connected groupoid. We start with
a connected groupid $\mathcal G$ with partial composition $\circ.$ Choose a 
morphism from each homset $\mathcal G(a,b)$
and is denote by $h_{ab}$ in such a way that $h_{ba} = {h_{ab}}^{-1}$ and $h_a 
= 1_a.$ Define an addition $+$ on $\mathcal G$ 
for $f\in \mathcal G(a, b), g\in \mathcal G(c,d)$ as follows
$$ f+g = f\circ h_{bc} \circ g .$$
 Clerarly $+$ is an associative binary operation on $\mathcal G$ and for each 
 $f\in \mathcal G(a,b)$ 
 
\begin{center}
\begin{equation*}
\begin{split}
f+h_{ab} & = f\circ h_{ba} \circ h_{ab}\\
& = f\circ{ h_{ab}}^{-1}\circ h_{ab}\\
& = f\circ 1_b\\
& = f\\
h_{ab} +f & = h_{ab} \circ h_{ba}\circ f\\
         & =  h_{ab} \circ {h_{ab}}^{-1}\circ f\\
         & = 1_a \circ f\\
         & = f\\
\end{split}
\end{equation*}
\end{center}

$f+h_{ab}  = h_{ab}+f  = f\;\; and \;\;e(f) = h_{ab} \;\; \forall\;\; 
f\in \mathcal G $

\begin{center}
\begin{equation*}
\begin{split}
f +(h_{ab}f^{-1} h_{ab})& = f \circ 1_b \circ f^{-1}\circ h_{ab}\\
          & = h_{ab}\\
          (h_{ab}f^{-1} h_{ab})+f & = h_{ab}\circ f^{-1}\circ 1_a \circ f\\
          & = h_{ab}
\end{split}
\end{equation*}
\end{center}
$$f +(h_{ab}f^{-1} h_{ab}) =  (h_{ab}f^{-1} h_{ab})+f = h_{ab} = e(f)$$
$$ -f = h_{ab}f^{-1} h_{ab} $$
 Thus the connected groupoid $\mathcal G$ is a generalized group with respect to 
 the addition we defined above. \\
 
Next we recall the generalized rings which was introduced by Molaei. 

\begin{defn}(cf.\cite{mrmolaei}) A generalized ring $R$ is a nonempty set $ R$ with 
two different operations 
adition and multiplication
 denoted by $`+`\; and\; `\times`$respectively in which $(R,+)$ is a 
generalized group and satisfies the  following conditions.
 \begin{enumerate}
  \item multiplication is an associative binary operation.
  \item for all $x,\; y,\; z\;\in R \quad x(y+z) = xy+xz$ and $(x+y)z = xz+yz$
 \end{enumerate}
 \end{defn}
 Note that in a generalized ring $R,$ $e(ab) =e(a)e(b)$ for all $a, b\;\in R.$
 \begin{exam}
 $\mathbb{R^2}$ with the operations $(a, b) +(c, d) = (a,d)$ and $ (a, b)(c, d)=(ac, 
bd)$ 
is a generalized ring.
 \end{exam}
   
\section{Generalized modules}
Let $M$ be a generalized group and $R$ be a ring, in the following we proceed to define a 
generalized module using the generalized group $M$. 
\begin{defn}
Let $R$ be a ring with unity. A generalized group $M$ is said to be (left) 
generalized $R$
module if for each element
 $r$ in $R$ and each $m$ in $M$ we have a product $rm$ in $M$ such that  for 
$r, s \in\; R \;\; and\;\; m, n \in M $
 
 \begin{enumerate}
  \item (r+s)m \;=\; rm+sm
  \item r(m+n)\; =\; rm+ rn
  \item r(sm)\; =\; (rs)m
  \item  $r e(m) = e(m)$ \;\;for all $r\in R$ and $m\in M$
  \item 1.m = m
 \end{enumerate}

 \end{defn} 
 \begin{prop}
  Let $M$ be an $R-$ module then $M \times M$ is a generalized $R$ module with  
the operations  
   $$(x, y)+(m, n) = (x, y+n)  $$ 
  $$r(x, y) = (x, ry)$$ and for each $x \in M$ the subset $M_x = \{ (x, y): y\in 
M\}$ is an $R$ module.
  \end{prop}
  \begin{proof}
   It can be seen that
  $(M, +)$ is a generalized group in which for all $(m, n) \in M,\;\; e(m, n)= 
(m, 0)$ and $(m, n)^{-1} = (m, -n).$ To show that $M$ is a generalized $R$ 
module consider $r, s \in R$ and $(x,y),(m,n) \in M\times M,$
  
\begin{enumerate}
\item \begin{equation*}
\begin{split}
(r+s)(x,y) & = (x,(r+s)y)\\
& = (x, ry+rs)
\end{split}
\end{equation*}
 \begin{equation*}
\begin{split}
r(x,y) +s(x,y) & = (x, ry)+(x,rs) \\
& =(x, ry+rs)
\end{split}
\end{equation*}
$ie;\;\;\;\;\;(r+s)(x,y)= r(x,y) +s(x,y)$ axiom(1)is satisfied

\item \begin{equation*} 
\begin{split}
r[(x,y)+(m,n)] & = r[(x,y+n)] \\
     & = (x, r(y+n))\\
     & = (x, ry+rn)    
\end{split}
\end{equation*}
\begin{equation*} 
\begin{split}
r(x,y) +r(m,n) & = (x, ry)+(m, rn)\\
     & = (x, ry+rn)
\end{split}
\end{equation*}

hence $r[(x,y)+(m,n)]= r(x,y) +r(m,n)$ axiom(2)is satisfied
\item \begin{equation*} 
\begin{split}
rs(x,y)  & = (x, rsy)\\
     & = (x, r(sy))\\
     & = r(x,sy)\\
     & = r(s(x,y))
\end{split}
\end{equation*}
$rs(x,y) = r(s(x,y) )$
axiom(3)is satisfied
\item \begin{equation*} 
\begin{split}
 re(x,y) & = r(x,0) \\ & = (x,r0)\\ & = (x,0)\\ & = e(x, y)
\end{split}
\end{equation*} 
$ re(x,y) = e(x,y)$ axiom(4)is satisfied

\item $$1\cdot(x, y)  = (x,y)\quad \forall (x,y)\in M$$

\end{enumerate}
To show that for each $x \in M$ the subset $M_x = \{ (x, y): y\in 
M\}$ is an $R$ module it is enough to show that $M_x$ is an abelian group
and the scalar multipication is closed. 
Let $m+n =(x, y)$ and $n= (x, z)$ be elements in $M_x$ then $$ m+n = (x,y)+(x,z) = (x, y+z)$$ and
$$ n+m = (x,z)+(x,y) = (x, z+y) = (x,y+z) $$
Therefore $$m+n = n+m$$ and $+$ is a commutative binary operation on $M_x.$  Forevery $m\in M_x$ 
$(x,y)+(x,0)= (x,0) +(x,y) = (x,y)$ and $(x,y)^{-1}= (x,-y).$ Thus $M_x$ is an abelian subgroup of$M$
and for any $r\in R\;\;\; r(x, y) = (x,ry) \in M_x.$ Hence $M_x$ is a $R$ module.
\end{proof}
Thus we have the following example for a generalized $\mathbb R$ module.

 \begin{exam}
   Consider $M = \mathbb R \times \mathbb R$ with the following operations 
  
  $$(x, y)+(m, n) = (x, y+n)  $$ 
  $$r(x, y) = (x, ry)$$
  is a generalized $R$ module.

 \end{exam}
 \begin{thm}
 If $M$ is a generalized $R$ module then 
 \begin{enumerate}
 \item$ e(rm) = re(m)$ for all $r \in R,\; 
m\in M$ 
\item $(rm)^{-1} = r(m^{-1})$
\end{enumerate}
\end{thm}

 \begin{proof}
 
  Let $r \in R$ and $m \in M$,
  \begin{enumerate}
   \item 
    $$ rm + re(m) = r(m+e(m)) = rm$$
  
  $$ re(m) + rm = r(e(m)+m) = rm$$
    Hence $e(rm) = re(m)$
  
  \item 
  $$rm + rm^{-1} = r(m+m^{-1}) = re(m)$$
  $$ rm^{-1} + rm = r (m^{-1} +m ) = re(m)$$
  
  \end{enumerate}
 \end{proof}

\begin{defn}
 Let $M$ and $N$ be two generalized $R-$modules. A function $f: M \rightarrow N 
$ is called generalized module 
 homomorphism if $$ f(m+ n)\; = f(m) +f(n)\;\; for\;all\; m,\; n\in M $$
 $$ f(rm) = rf(m)\;\; for\;all\; r\in R,\; m \in M $$
\end{defn}

\begin{defn}
 Let $M$ be a generalized $R$ module and $N\subset M $ is said to be 
generalized submodule of $M$ if  $N$ is a generalized subgroup of $G$ and for 
each $r\in R 
\; and \; m \in N$
 $rm \in N $ 
 \end{defn}
 
 \begin{prop}
Let $M$  be a generalized $R$-module  which  is also a normal generalized group. 
Then the set of identity elements in $M$ is a submodule of 
$M$ and we call it the zero submodule or trivial submodule.
\end{prop}

\begin{proof}
We denote the set of identity elements in $M$ by $$e(M)=\{e(x)\;:\; x\in M\}$$
First we show that $e(M)$ is a generalized subgroup of $M.$ For $m,\; n\;\in 
e(M)$ there exists $x,\; y\;\;\in M$
such that $m = e(x)$ and $n = e(y).$ $$ mn^{-1} = e(x)e(y)^{-1} = e(x)e(y) = 
e(xy)$$ hence 
 $xy^{-1} \in e(M)$ and $e(M)$ is a generalized subgroup of $M.$ \\
 Ler $r$ be an element in the ring $R$ and $$rm = re(x) = e(x)$$
hence $rm \in M$ for all $r\in R$ and $m \in M.$ Thus set of identitt elements of $M$ 
form a generalized submodule of $M.$ 
\end{proof}

Every nonzero generalized module $M$ contains at least one submodule $M$ 
itself. 

\begin{thm} Let $R$ is a ring and $M$ and $N$ are normal generalized $R$ modules. 
If $f: M\rightarrow N$ is a generalized $R-$ module homomorphism, 
then   $$ ker f = \{x: x \in M,\; f(x)\in e(N) \}$$ where $e(N)$ 
denotes the set of identities elements of $N,$
is a submodule of $M$ and $Im f = \{f(x):\;\; x\in M \}$ is a submodule of $N.$ 
\end{thm}

\begin{proof}
 Let $M$ and $N$ be a generalized modules over a ring $R$ and $f$ is a 
  homomorphism between $M$ and $N.$ To show that  $ ker f $ 
is a submodule of $M$ it is suffices to prove that it is a generalized 
subgroup of $M$ and is closed under scalar multipication.\\ Let $ x, y \in ker f $ then  $ f(x) = e(n) \; and\; f(y)= e(n')$ for some $n, \; n'\in N$
Now consider;
\begin{equation*}
 \begin{split} 
f(xy^{-1}) & = f(x)f(y^{-1})\\ & = f(x)(f(y))^{-1}\\ & = e(n)e(n')^{-1} \\ 
& = e(n)e(n')\\ & = e(nn')
\end{split}
\end{equation*}
 hence $xy^{-1} \in ker f$ and $ ker f$ is a generalized subgroup of $M.$ \\
for any $r\in R$ 
\begin{equation*}
 \begin{split}
  f(rx) & = rf(x)\\ & = re(n)\\ &= e(n)\\
  rx \in ker f \; for\; each\;\;r\in R.
 \end{split}
\end{equation*}

Therefore $ker f$ is a submodule of $M$ for every $m\in M.$ \\
Similarly we can prove that $Im f$ is a submodule of $N.$ For;
let $x, y \in Im f$ then there exists $m, n, \in M$ such that
$f(m) = x$ and $f(n) = y.$ Now consider,
\begin{equation*}
 \begin{split}
  xy^{-1} & = f(m)f(n)^{-1}\\
          & = f(m) f(n^{-1}) \\
          & = f(mn^{-1})\\
          and\;\; mn^{-1} \in M \;\; hence \;\; xy^{-1}\in Im f.
 \end{split}
\end{equation*}
Hence $Im f$ is a generalized subgroup of $N$ and for any $r \in R,$\\
\begin{equation*}
 \begin{split}
  rx & = rf(m)\\ & = f(rm),\;\; rm\in M
 \end{split}
\end{equation*}
Hence $rx \in Im f$ for any $r \in R$ and $x \in M.$ Therefore
$Im f$ is a generalized submodule of $N.$
\end{proof}

\begin{thm}
If $M$ is a generalized $R$-module and if there exists $x \in M$ such that $M = 
Rx$ then $M$ is a module over $R$ 
 \end{thm}

\begin{proof}
  First we show that $M$ is an abelian group with respect to the generalized 
group operation. For any $m,\; n\; \in M$ there exists $r,\; s\; \in R$
  
  such that $$ m = rx\quad and \quad n = sx $$
  Consider $$ rx + sx = (r+s)x = (s+r)x = sx+rx $$ 
  $$ ie;\;\;\; m+n = n+m \;\;\; for\; all\; \; m, n \in M.$$ Hence $M$ is an 
abelian generalized group hence is an abelian group.\\
  Therefore $M$ is a $R-$ module.
 \end{proof}
\begin{note}
 The generalized modules and their homomorphisms form a category in which 
objects are
  the generalized modules and morphisms are their homomorphisms denoted by 
$\mathcal{GM}.$ 
\end{note}
\begin{res}
If $M$ and $N$ be two generalized modules over a ring $R$ then their cartesian product 
$M\times N$ defined by $$ M\times N = \{ (m,n):\;\;m\in M,\; \; n\in N\}$$ 
is a generalized module over $R$ with respect to the component wise operations.
$ie;$ for any $(m, n),\; (x,y)\;\in M\times N$ and $r\in R$ the addition and scalar multipication is given by 
$$(m,n)+(x,y) = (m+x,n+y)$$  $$ r(m,n) = (rm, rn)$$
\end{res}

\begin{defn}
A groupoid $\mathcal G$ is a generalized module groupoid over $R$ 
 if it has a generalized module structure over $R$ and it satisfies the  
following conditions.
\begin{enumerate}
\item $\mathcal G$ is a generalized group groupoid.
\item For each $r\in R$ the mapping $$\eta_r : \mathcal G\rightarrow \mathcal 
G$$
defined by $\eta_r(g) = rg$ is a functor on $\mathcal G$.
\\
$ie;$ for any composable morphisms in $\mathcal G$ and any $r\in R$ we should 
have
$$\quad r(g\circ h) = rg \circ rh.$$ 
\end{enumerate}
\end{defn}

\begin{exam}
Let $M$ be a generalized module over a ring $R.$ Then $M\times M$ is a generalized module  
groupoid over $R$ with object set $M.$ It follows from the example(6) that 
$M\times M$ with operation $(x,y)+(m,n) = (x+m,y+n)$
is a generalized group groupoid. The cartesian product of two generalized modules are
again a generalized module.
To show that $M\times M$ is a generalized module groupoid over $R$
it is enough to prove that for any $r\in R$ the map $\eta_r: M\times M 
\rightarrow M\times M$ defined by $\eta_r(x, y) = (rx, ry)$ is a functor. For,
let $x$ be any object in the category $M\times M$ then 
$$\eta_r(1_x) = \eta_r(x,x) = (rx, rx) = 1_{\eta_r(x)}. $$
Consider  two composable morphisms  $(x,y), (y, z)  $ in $M\times M$ then 
\begin{equation*}
\begin{split}
\eta_r[(x,y)\circ(y, z)] & = \eta_r[(x, z)]\\
 & = (rx, rz)\\
 \eta_r(x,y)\circ \eta_r(y, z)& = (rx, ry)\circ (ry, rz)\\
 & = (rx, rz)\\
 \eta_r[(x,y)\circ(y, z)]& = \eta_r(x,y)\circ \eta_r(y, z)
\end{split}
\end{equation*}
hence $M\times M$ is a generalized module groupoid over $R.$
\end{exam}

\begin{defn}
   Let $M$ and $N$ be two generalized module groupoids over a ring $R.$ A 
homomorphism $f: M\rightarrow N$
   of generalized module groupoids is a functor of underlying groupoids 
preserving generalized module structure.  
\end{defn}

Note That the generalized module groupoids and their homomorphisms form a category 
denoted by $\mathcal G\mathcal M\mathcal G $
\begin{prop}
 There is a functor from the category $\mathcal G\mathcal M$ of generalized 
modules to  the category $\mathcal G\mathcal M\mathcal G $ of generalized module groupoid
\end{prop}
\begin{proof}
 Let $M$ be a generalized module over a ring $R.$ Then it can be seen that 
cartesian product $M\times M$ is  a generalized module groupoid over $R.$ 
If $f: M_1 \rightarrow M_2$ is a homomorphism of 
 generalized modules then define $ F :\mathcal G\mathcal M \rightarrow  \mathcal 
G\mathcal M\mathcal G $ by
 $$F(M) = M\times M$$ and $$ F(f): M_1 \times M_1 \rightarrow M_2\times M_2$$ 
$$F(f)(m,n) = (f(m), f(n))$$ and 
 it can be seen that $F(f)$ is a functor between $M_1\times M_1$ and $M_2\times 
M_2$ so that $F(f)$ is a morphism  in the category of generalized module groupoid.
 Now we prove that $F$ is a functor between the category of generalized modules 
and the generalized module groupoids. For; let for any vertex $M\in\mathcal G \mathcal M$ we have,
$$F(1_M)(m,n) = (m,n)$$ $$ F(1_M) = 1_{F(M)}$$
 let $f : M\rightarrow N $ and $g: N\rightarrow P$ be two composable 
homomorphisms in $\mathcal{GM}$ then
 \begin{equation*}
  \begin{split}
   F(fg)(m,n) & = (fg(m), fg(n)) \\
              & = ((g(f(m)), g(f(n)))\\
   F(f)F(g)(m,n) & = F(g) (F(f)(m,n))  \\
   & = F(g)(f(m),f(n))\\
   &= (g(f(m)), g(f(n)))\\
    F(fg)(m,n) & =  F(f)F(g)(m,n) \;\; \forall (m,n) \in M\times M\\
   Hence\;\;\;   F(fg)= F(f)F(g)
     \end{split}
 \end{equation*} 
\end{proof}

\begin{prop}
 Let $\{ M_i\;\; i\in I\}$ be a family of generalized $R$ module groupoids.
 Then $M=(\nu M,M)$ where $\nu M = \prod \nu M_i $ and $M = \prod M_i$ is a generalized $R$ module groupoid.
\end{prop}

\begin{proof}
 $\prod \nu M_i$ has elements $(m_i)_{i\in I}$ where $m_i\in \nu M_i$ and morphisms $(f_i)_{i\in I}$ from 
 $dom\, f_i$ to $cod\,f_i$ in $M_i$. The composition of morphisms is 
 $$(f_i)_{i\in I} \cdot (g_i)_{i\in I}=(f_i\cdot g_i)_{i\in I}$$
 whenever $f_i$ and  $g_i$ re composable morphisms in $M_i$. Since each $M_i$ is a groupoid each  morphism $f_i$ admits an inverse $f_i^{-1}$, thus the product 
 $M=(\prod \nu M_i, \prod M_i),\,i\in I$ is a groupoid. Moreover the product $M = \prod M_i$ has a structure of generalized module with respect to component wise operations 
 $$(f_i)_{i\in I} + (g_i)_{i\in I} = (f_i+g_i)_{i\in I}\;\; and \;\; r(f_i)_{i\in I} = (rf_i)_{i\in I}$$    
thus $M$  is a generalized group groupoid. 
It remains to show that the map $\eta_r: M \rightarrow M$ by 
$$ \eta_r(m_i)_{i\in I} = (rm_i)_{i\in I}$$
is a  functor on $ M.$ For each  $ (x_i)_{i\in I} \in \nu M $ consider
\begin{equation*}
\begin{split}
 \eta_r (1_{(m_i)})_{i\in I} & = (r1_{(m_i)})_{i\in I}\\
 & = (1_{(rm_i)})_{i\in I}\;\;\;\;( each\; M_i\; is\; a\; generalized\; module\; groupid)\\
 & = 1_{\eta_r(m_i)_{i\in I}}
\end{split}
\end{equation*}                                                 
let $(m_i)_{i\in I}, (n_i)_{i\in I}$ are two composable morpisms in 
$M,$ then 
\begin{equation*}
\begin{split}
  \eta_r((m_i)_{i\in I}\circ (n_i)_{i\in I}) & = r((m_i)_{i\in I}\circ (n_i)_{i\in I})\\
     & = r(m_i \circ n_i)_{i\in I} \\
     & = (r (m_i \circ n_i)_{i\in I})\\
     & = (rm_i \circ rn_i)_{i\in I}\\
     & = (rm_i)_{i\in I}\circ (rn_i)_{i\in I}\\
     \eta_r(m_i)_{i\in I} \circ \eta_r(n_i)_{i\in I} & = (rm_i)_{i\in I}\circ (rn_i)_{i\in I}\\
     \eta_r((m_i)_{i\in I}\circ (n_i)_{i\in I}) & = \eta_r(m_i)_{i\in I} \circ \eta_r(n_i)_{i\in I}
 \end{split}
\end{equation*}
\end{proof}

\end{document}